\documentclass[a4paper,11pt]{amsart}

\usepackage[english]{my-shortcuts}
\usepackage{srcltx,algorithm,algorithmic}
\usepackage{graphicx}   

\title[A prediction bound for the LASSO]{On prediction with the LASSO when the design is not incoherent}
\author{St\'ephane Chr\'etien} \thanks{Laboratoire de Math\'ematiques, UMR 6623,
Universit\'e de Franche-Comt\'e, 16 route de Gray,
25030 Besancon, France. Email: stephane.chretien@univ-fcomte.fr}

\begin{document}
\maketitle


%
%
%

\begin{abstract}
The LASSO estimator is an $\ell_1$-norm penalized least-squares estimator, which was introduced for variable selection in the linear model. When the design matrix satisfies, e.g. the Restricted Isometry Property, or has a small coherence index, the LASSO estimator has been proved to recover, with high probability, the support and sign pattern of sufficiently sparse regression vectors. Under similar assumptions, the LASSO satisfies adaptive prediction bounds in various norms. The present note provides a prediction bound based on a new index for measuring how favorable is a design matrix for the LASSO estimator. We study the behavior of our new index for matrices with independent random 
columns uniformly drawn on the unit sphere. Using the simple trick of appending such a random 
matrix (with the right number of columns) to a given design matrix, we show that a prediction bound 
similar to \cite[Theorem 2.1]{CandesPlan:AnnStat09} holds without any constraint on the 
design matrix, other than restricted non-singularity. 
\end{abstract}


\bigskip

\section{Introduction}
Given a linear model 
\bea 
y=X\beta +\epsilon
\eea
where $X\in \mathbb R^{n\times p}$ and $\epsilon$ is a random vector with gaussian distribution $\mathcal N(0,\sigma^2 I)$ 
the LASSO estimator is given by 
\bea 
\label{lasso}
\hat{\beta} & = & {\rm argmin}_{\beta \in \mathbb R^p} \frac12 \|y-X\beta\|_2^2+\lb \|\beta\|_1.
\eea
This estimator was first proposed in the paper of Tibshirani \cite{Tibshirani:JRSSB96}. 
The LASSO estimator $\hat{\beta}$ is often used in the high dimensional setting where $p$ is much 
larger than $n$. As can be expected, when $p\gg n$, 
estimation of $\beta$ is hopeless in general unless some 
additional property of $\beta$ is assumed. In many practical situations, it is considered relevant 
to assume that $\beta$ is sparse, i.e. has only a few nonzero components, or 
at least compressible, i.e. the magnitude of the non zero coefficients decays with high rate. It is 
now well recognized that the $\ell_1$ penalization of the likelihood often promotes sparsity under 
certain assumptions on the matrix $X$. We refer the reader to the book \cite{BuhlmannVanDeGeer:Springer11}
and the references therein for a state of the art presentation of the LASSO and the tools 
involved in the theoretical analysis of its properties. 
One of the main interesting properties of the LASSO estimator is that it is a solution of 
a convex optimization problem and it can be computed in polynomial time, 
i.e. very quickly in the sense of computational complexity theory. This makes a big 
difference with other approaches based on variable selection criteria like AIC
\cite{Akaike:IEEEAC74}, BIC \cite{Schwarz:AnnStat78}, 
Foster and George's Risk Inflation Criterion \cite{FosterGeorge:AnnStat94}, etc, which 
are based on enumeration of the possible models, 
or even with the recent proposals of Dalalyan, Rigollet and Tsybakov 
\cite{RigolletTsybakov:AnnStat11}, \cite{DalalyanTsybakov:JCSS12}, although enumeration 
is replaced with a practically more efficient Monte Carlo Markov Chain algorithm. 

In the problem of estimating $X\beta$, i.e. the prediction problem, 
it is often believed that the price to pay for reducing the variable selection approach 
to a convex optimization problem is a certain set of assumptions on the design matrix $X$
\footnote{Conditions for model selection consistency are given in e.g. \cite{ZhaoYu:JMLR06}, 
and for exact support and sign pattern recovery with finite samples and $p\gg n$, 
in \cite{CandesPlan:AnnStat09}, \cite{Wainwright:IEEEIT09}.}. 
One of the main contributions of \cite{RigolletTsybakov:AnnStat11} is that no particular 
assumption on $X$ is required for the prediction problem, as opposed to the known results 
concerning the LASSO such that \cite{BuneaTsybakovWegkamp:EJS07}, 
\cite{BickelRitovTsybakov:AnnStat09}, \cite{CandesPlan:AnnStat09} and \cite{vdGeer:EJS09}, 
and the many references cited in these works. 

An impressive amount of work has been done in the recents 
years in order to understand the properties of $\hat{\beta}$ unded various assumptions on $X$. 
See the recent book by P. Buhlmann and S. Van de Geer \cite{BuhlmannVanDeGeer:Springer11} for 
the state of the art. Two well known assumptions on the design matrix are 
\begin{itemize}
\item small Coherence $\mu(X)$
\item small Restricted Isometry Constant $\delta(X)$
\end{itemize}
where the Coherence $\mu(X)$ is defined as 
\bean 
\mu(X) & = & \max_{j,j^\prime} |X_j^tX_{j^\prime}|,
\eean 
and the Restricted Isometry Constant $\delta(X)$ is the smallest $\delta$ such that 
\bea
\label{rip}
(1-\delta) \|\beta_T\|_2 \le \|X_T \beta_T\|_2 \le (1+\delta) \|\beta_T\|_2
\eea 
for {\bf all} subset $T$ with cardinal $s$ and all $\beta \in \mathbb R^p$. Other 
conditions are listed in \cite{BuhlmannVandeGeer:EJS09}; see Figure 1 in that paper for a diagram 
summarizing all relationships between them. The Restricted Isometry property is very 
stringent and implies almost other conditions. Moreover, the Restricted Isometry Constant 
is NP-hard to compute for general matrices. On the other hand, the Coherence only 
requires of the order $np(p-1)$ elementary operations. However,  
it was proved in \cite{CandesPlan:AnnStat09} that a small coherence, say of 
the order of $1/\log(p)$, is sufficient to prove a property very close to the Restricted Isometry
Property: (\ref{rip}) holds for a large proportion of subsets $T\subset \{1,\ldots,p\}$, 
$|T|=s$ (of the order $1-1/p^\alpha$, $\alpha>0$). This result was later refined in 
\cite{ChretienDarses:SPL12} with better constants using the recently discovered 
Non-Commutative deviation inequalities \cite{Tropp:FOCM12}. Less stringent properties 
are the restricted eigenvalue, the irrepresentable and the compatibility properties. 

The goal of this short note is to show that, using a very simple trick, one can prove prediction bounds 
similar to \cite[Theorem 2.1]{CandesPlan:AnnStat09} without any assumption on the design matrix $X$ at the 
low expense of appending to $X$ a random matrix with independent columns uniformly distributed on the 
sphere. 

For this purpose, we introduce a new index for design matrices, denoted by 
$\gamma_{s,\rho_-}(X)$ that allows to obtain novel adaptive bounds on the prediction error. 
This index is defined for any $s\le n$ and $\rho_- \in (0,1)$ as 
\bea 
\gamma_{s,\rho_-}(X) & = & \sup_{v \in B(0,1)} \inf_{I \subset \mathcal S_{s,\rho_-}} \|X_I^t v\|_{\infty},
\eea
where $\mathcal S_{s,\rho_-}(X)$ is the family of all $S$ of $\{1,\ldots,p\}$ with 
cardinal $|S|=s$, such that $\sigma_{\min}(X_S)\ge \rho_-$. The meaning of the index $\gamma_{s,\rho_-}$ is the following: for any $v\in \R^n$, we look for the "almost orthogonal" family inside the set of columns of $X$ with 
cardinal $s$, which is the most orthogonal to $v$ . 

One major advantage of this new parameter is that imposing the condition that $\gamma_{s,\rho_-}$ 
is small is much less stringent than previous criteria required in the litterature. 
In particular, many submatrices of $X$ may be very badly conditioned or even singular without altering the 
smallness of $\gamma_{s,\rho_-}$. 
Computing the new index $\gamma_{s,\rho_-}(X)$ for random matrices with 
independent columns uniformly distributed on the sphere \footnote{or equivalently, 
post-normalized Gaussian i.i.d. matrices with components following $\mathcal N(0,1/n)$.}, 
shows that a prediction bound involving $\gamma_{s,\rho_-}(X)$ can be obtained which is of the same order 
as the bound of \cite[Theorem 2.1]{CandesPlan:AnnStat09}. 

One very nice property of the index $\gamma_{s,\rho_-}$ is that it decreases after the operation appending
any matrix to a given one. As a very nice consequence of this observation, the results obtained for random matrices
can be extended to any matrix $X$ to which a random matrix is appended. This trick can be used to 
prove new prediction bounds for a modified version of the LASSO obtained by appending a random matrix to 
any given design matrix. This simple modification of the LASSO retains the fundamental property of being polynomial
time solvable unlike the recent approaches based on non-convex criteria for which no computational complexity analysis 
is available. 

The plan of the paper is as follows. In Section \ref{MainRes} we present the index $\gamma_{s,\rho_-}$ for 
$X$ and provide an upper bound on this index for random matrices with independent columns 
uniformly distributed on the sphere, holding with high probability. 
Then, we present our prediction bound in Theorem \ref{main}: 
we give a bound on the prediction squared error 
$\|X(\beta-\hat{\beta}\|_2^2$ which depends linearly on $s$. This result is similar in spirit to 
\cite[Theorem 1.2]{CandesPlan:AnnStat09}. The proofs of the above results are given in Section \ref{Pf}. 
In Section \ref{GaussAppend}, we show how these results can be applied in practice to any problem 
with a matrix for which $\gamma_{s,\rho_-}$ is unknown by appending to $X$ an $n\times p_0$ random matrix 
with i.i.d. columns uniformly distributed on the unit sphere of $\R^n$ and with only $p_0=O(n^{\frac32}\rho_+)$ columns. An appendix contains the proof of some intermediate results.

\subsection{Notations and preliminary assumptions}
A vector $\beta$ in $\R$ is said to be $s$-sparse if exactly $s$ of its components are different from zero. 
Let $\rho_-$ be a positive real number. 
In the sequel, we will denote by $\mathcal S_{s,\rho_-}(X)$ 
the family of all index subsets $S$ of $\{1,\ldots,p\}$ 
with cardinal $|S|=s$, such that for all $S\in \mathcal S_{s,\rho_-}$, $\sigma_{\min}(X_S)\ge \rho_-$. 

\section{Main results}
\label{MainRes}
\subsection{A new index for design matrices}
\begin{defi}
\label{defgam}
The index $\gamma_{s,\rho_-}(X)$ associated with the matrix $X$ in $\mathbb R^{n\times p}$ is defined by 
\bea 
\gamma_{s,\rho_-}(X) & = & \sup_{v \in B(0,1)} \inf_{I \subset \mathcal S_{s,\rho_-}} 
\|X_I^t v\|_{\infty}.
\eea
\end{defi}
An important remark is that the function $X\mapsto \gamma_{s,\rho_-}(X)$ is nonincreasing in the sense that if 
we set $X^{\prime\prime}=[X,X^\prime]$, 
where $X^\prime$ is a matrix in $\mathbb R^{n \times p^\prime}$, then $\gamma_{s,\rho_-}(X)\ge \gamma_{s,\rho_-}(X^\prime)$.

Unlike the coherence $\mu(X)$, for fixed $n$ and $s$, the quantity $\gamma_{s,\rho_-}(X)$ is very small for 
$p$ sufficiently large, at least for random matrices such as normalized standard Gaussian matrices as shown 
in the following proposition. 
\begin{prop}
\label{gamtheo}
Assume that $X$ is random matrix in $\mathbb R^{n\times p}$ with i.i.d. columns with 
uniform distribution on the unit sphere of $\R^n$. 
Let $\rho_-$ and $\epsilon\in(0,1)$, $C_\kappa\in (0,+\infty)$ and $p_0\in \{\lceil e^{\frac6{\sqrt{2\pi}}}\rceil,\ldots,p\}$. 
Set 
\bean 
K_\epsilon & = & \frac{\sqrt{2\pi}}{6}\left(\left(1+C_\kappa \right) 
\log\left(1+\frac2{\epsilon}\right)+C_\kappa+\log\left(\frac{C_\kappa}{4}\right)\right).
\eean 
Assume that $n$, $\kappa$ and $s$ satisfy 
\bea
& n \ge 6, \\
\nonumber \\
\label{kappa}
& \kappa = \max\left\{ 4 e^{-2(\ln(2)-1)},\frac{4 e^3}{(1-\rho_-)^2} \ \left( \frac{(1+K_\epsilon)(1+C_{\kappa})}{c(1-\epsilon)^4}\right)^2 \log^2(p_0) \log(C_\kappa n)\right\}, \\
\nonumber \\
\label{n}
& \displaystyle{\frac{\max\left\{\kappa s, 2\times 36 \times 3 \times 3,\exp((1-\rho_-)/2) \right)}{C_{\kappa}}} 
\le n \le \min \left\{\left(\frac{p_0}{\log(p_0)}\right)^2,\frac{\exp\left(\frac{1-\rho_-}{\sqrt{2}}p_0 \right)}{C_\kappa}\right\}.
\eea 
Then, we have 
\bea 
\gamma_{s,\rho_-}(X) & \le & 80\ \frac{\log(p_0)}{p_0}
\eea
with probability at least $1-5 \ \frac{n}{p_0\ \log(p_0)^{n-1}}-9\ p_0^{-n}$. 
\end{prop}
\begin{rem}
Notice that the constraints (\ref{kappa}) and (\ref{n}) together imply the following constraint on $s$:
\bean 
\label{s}
& s \le & C_{sparsity} \ \frac{ n}{ \log^2(p_0) \log(C_\kappa n)}
\eean 
with 
\bean 
C_{sparsity} & = & \frac{c^2(1-\rho_-)^2(1-\epsilon)^8}{4e^3} \ \frac{C_{\kappa}}{(1+K_{\epsilon})^2(1+C_{\kappa})^2}.
\eean 
\end{rem}

\subsection{A bound of $\|X(\beta-\hat{\beta})\|_2^2$ based on $\gamma_{s,\rho_-}(X)$}
In the remainder of this paper, we will assume that the columns of $X$ are $\ell_2$-normalized. 
The main result of this paper is the following theorem. 
\begin{theo}
\label{main}
Let $\rho_-\in (0,1)$. Let $\nu$ be a positive real such that 
\bea 
\label{nu}
\nu \ \gamma_{\nu n,\rho_-}(X) \le \frac{\rho_-\ \sigma_{\min}(X_S)}{n \ \max_{\stackrel{T\subset \{1,\ldots,p\}}{|T|\le n}}\sigma_{\max}(X_T)}.
\eea
Assume that $s\le \nu n$. Assume that $\beta$ has support  $S$ with cardinal $s$ and that  
\bea
\label{asslamb} 
& \lb \ge   \\
& \sg  \left(B_{X,\nu,\rho_-} \ \max_{\stackrel{T\subset \{1,\ldots,p\}}{|T|\le n}}\sigma_{\max}(X_T) 
\sqrt{2\alpha \ \log(p)+\log(2\nu n)}+ \sqrt{(2\alpha+1) \ \log(p)+\log(2)}
\right)  
\nonumber 
\eea 
with 
\bea 
B_{X,\nu,\rho_-} & = & 
\frac{\nu n \ \gamma_{\nu n,\rho_-}(X)}{\rho_- \ \sigma_{\min}(X_S)- \nu n \ \gamma_{\nu n,\rho_-}(X)\ 
\max_{\stackrel{T\subset \{1,\ldots,p\}}{|T|\le n}}\sigma_{\max}(X_T)}.
\eea
Then, with probability greater than $1-p^{-\alpha}$, we have 
\bea
\frac12 \|X(\hat{\beta}-\beta)\|_2^2 & \le & s \ C_{n,p,\rho_-,\alpha,\nu,\lb} 
\eea
with 
\bea
C_{n,p,\rho_-,\alpha,\nu,\lb} & = & \frac{\lb + \sg \sqrt{(2\alpha+1) \ \log(p)+\log(2)}}{
\rho_- \ \sigma_{\min}(X_S)}
\left(\sg \sqrt{2\alpha \ \log(p)+\log(2\nu n)}+\lb \right)
\eea
\end{theo}
\subsection{Comments}
Equation (\ref{nu}) in Theorem \ref{main} requires that 
\bea
\label{condgam}
\gamma_{\nu n,\rho_-}(X) & < & \rho_- \frac{\sigma_{\min}(X_S)}
{\nu n \ \max_{\stackrel{T\subset \{1,\ldots,p\}}{|T|\le n}}\sigma_{\max}(X_T)}. 
\eea 
Proposition \ref{gamtheo} proves that for random matrices with independent columns 
uniformly drawn on the unit sphere of $\mathbb R^n$ (i.e. normalized i.i.d. gaussian matrices), 
\bea 
\gamma_{s,\rho_-}(X) & \le & 80\ \frac{\log(p_0)}{p_0}
\eea
with high probability. The case of general design matrices can be treated using a simple trick. 
It will be studied in Section \ref{GaussAppend}. 

The main advantage of using the parameter $\gamma_{\nu n,\rho_-}(X)$ is that it allows $X$ to
contain extremely badly conditioned submatrices, a situation that may often occur in practice when 
certain covariates are very correlated. This is in contrast with the Restricted Isometry Property 
or the Incoherence condition, or other conditions often required in the litterature. 
On the other hand, the parameter $\gamma_{\nu n,\rho_-}(X)$ is not easily computable. 
We will see however in Section \ref{GaussAppend} how to circumvent this problem in practice by 
the simple trick consisting of appending a random matrix with $p_0$ columns to 
the matrix $X$ in order to ensure that $X$ satisfies (\ref{condgam}) with high 
probability.   

Finally, notice that unlike in \cite[Theorem 2.1]{CandesPlan:AnnStat09}, we make no  
assumption on the sign pattern of $\beta$. In particular, we do not require the sign pattern of 
the nonzero components to be random. Moreover, the extreme singular values of $X_S$ are not 
required to be independent of $n$ nor $p$ and the condition (\ref{condgam}) is satisfied for 
a wide range of configurations of the various parameters involved in the problem.

\section{Proofs} 
\label{Pf}
\subsection{Proof of Proposition \ref{gamtheo}}

\subsubsection{Constructing an outer approximation for $I$ in the definition of $\gamma_{s,\rho_-}$}

Take $v\in \mathbb R^n$. We construct an outer approximation $\tilde{I}$ 
of $I$ into which we be able to extract the set $I$. We procede recursively as follows: 
until $|\tilde{I}|=\min \{\kappa s,p_0/2\}$, for some positive real number $\kappa$ to be specified later, do 
\begin{itemize}
\item Choose $j_1={\rm argmin}_{j=1,\ldots,p_0} |\la X_{j},v\ra|$ and set $\tilde{I}=\{j_1\}$
\item Choose $j_2={\rm argmin}_{j=1,\ldots,p_0, \ j\not\in \tilde{I}} |\la X_{j},v\ra|$ and set $\tilde{I}=\tilde{I}\cup\{j_2\}$
\item $\cdots$
\item Choose $j_{k}={\rm argmin}_{j=1,\ldots,p_0, \ j\not\in \tilde{I}} |\la X_{j},v\ra|$ and set $\tilde{I}=\tilde{I}\cup\{j_k\}$. 
\end{itemize}

\subsubsection{An upper bound on $\|X^t_{\tilde{I}}v\|_\infty$}
If we denote by $Z_j$ the quantity $|\la X_{j},v\ra|$ and by $Z_{(r)}$ the $r^{th}$ 
order statistic, we get that 
\bean 
\|X^t_{\tilde{I}}v\|_\infty & = & Z_{(\kappa s)}.
\eean
Since the $X_j$'s are assumed to be i.i.d. with uniform distribution on the unit sphere of $\R^n$, 
we obtain that the distribution of $Z_{(r)}$ is the distribution of the $r^{th}$ order statistics
of the sequence $|X_j^tv|$, $j=1,\ldots,p_0$. 
By (5) p.147 \cite{Muirhead:AspectsMultAnal05}, $|X_j^tv|$ has density $g$ and CDF $G$ 
given by 
\bean
g(z) & = \frac1{\sqrt{\pi}} \frac{\Gamma\left(\frac{n}2\right)}{\Gamma\left(\frac{n-1}2\right)}
\left(1-z^2 \right)^{\frac{n-3}2}
\textrm{ and }
G(z) = & 2\ \int_{0}^z g(\zeta) \ d\zeta.
\eean
Thus,
\bean
F_{Z_{(r)}}(z) & = & \bP \left(B\ge r \right)
\eean
where $B$ is a binomial variable $\cal B\left(p_0,G(z)\right)$.
Our next goal is to find the smallest value $z_0$ of $z$ which satisfies 
\bea
\label{devorder}
F_{Z_{(\kappa s)}}(z_0) & \ge & 1-p_0^{-n}.
\eea
We have the following standard concentration bound for $B$ (e.g. \cite{Dubhashi:CUP09}): 
\bean 
\mathbb P \left(B \le (1-\epsilon) \mathbb E[B] \right) & \le &  \exp\left(-\frac12 \ \epsilon^2  \mathbb E[B]\right)
\eean 
which gives 
\bean 
\mathbb P \left(B \ge (1-\epsilon) p_0G(z) \right) & \ge &  1-\exp\left(-\frac12 \ \epsilon^2 p_0 G(z) \right)
\eean 
We thus have to look for a root (or at least an upper bound to a root) of the equation 
\bean 
G(z) & = & \frac1{\frac12 \ \epsilon^2 } \ \frac{n}{p_0} \ \log(p_0). 
\eean 
Notice that
\bean
G(z) & = & 2\ \frac1{\sqrt{\pi}} \frac{\Gamma\left(\frac{n}2\right)}{\Gamma\left(\frac{n-1}2\right)}
\ \int_{0}^z  \ \left(1-\zeta^2 \right)^{\frac{n-3}2} \ d\zeta, \\
& \ge & \frac1{\sqrt{\pi}} \frac{\Gamma\left(\frac{n}2\right)}{\Gamma\left(\frac{n-1}2\right)} \ z
\eean
for $z\le 1/\sqrt{2}$. By a straightforward application of Stirling's formula (see e.g. (1.4) 
in \cite{Qi:JIA10}), we obtain 
\bean 
\frac{\Gamma\left(\frac{n}2\right)}{\Gamma\left(\frac{n-1}2\right)} & \ge & \frac{e^{2\ln(2)}}{2} \ \frac{(n-3)^{3/2}}{(n-2)^{1/2}}. 
\eean 
Thus, any choice of $z_0$ satisfying
\bea
\label{left} 
z_0 & \ge &  \frac{2\ \sqrt{\pi}}{e^{2\ln(2)}} \ \frac{(n-2)^{1/2}}{(n-3)^{3/2}}\ \frac1{\frac12 \ \epsilon^2} \ 
\frac{n}{p_0} \ \log(p_0)
\eea
is an upper bound to the quantile for $(1-\epsilon) p_0 G(z_0)$-order statistics at level $p_0^{-n}$. 
We now want to enforce the constraint that 
\bean
(1-\epsilon) p_0 G(z_0) & \le & \kappa s. 
\eean 
By again a straightforward application of Stirling's formula, we obtain 
\bean
G(z) & \le & \frac1{\sqrt{\pi}}\ \frac{ e^2}{2} \ \frac{(n-3)^{3/2}}{(n-2)^{1/2}}\ z
\eean 
for $n\ge 4$. Thus, we need to impose that 
\bea
\label{right}
z_0 & \le & \frac{2 \sqrt{\pi}}{e^2} \ \frac{(n-2)^{1/2}}{(n-3)^{3/2}} \ \frac{\kappa s}{(1-\epsilon) p_0}. 
\eea 
Notice that the constraints (\ref{left}) and (\ref{right}) are compatible if 
\bean 
\kappa  & \ge & \frac{4}{e^{2(\ln(2)-1)}} \ \frac{1-\epsilon}{\epsilon^2} \ 
\frac{n}{s} \ \log(p_0) .
\eean 
Take $\epsilon=1-\frac1{n/s \log(p_0)}$ and obtain
\bean
\label{quantunifsph0} 
\bP \left( \|X^t_{\tilde{I}}v\|_\infty \ge \frac{8\ \sqrt{\pi}}{e^{2\ln(2)}} \ \frac{(n-2)^{1/2}}{(n-3)^{3/2}} \ \frac{n}{p_0} \ \log(p_0) \right) & \le & p_0^{-n}
\eean
for 
\bean 
\kappa & = & \frac{4}{e^{2(\ln(2)-1)}}  
\eean 
for any $p_0$ such that $n/s \log(p_0)\ge \sqrt{2}$, which is clearly the case as soon as $p_0\ge e^{\frac6{\sqrt{2\pi}}}$
for $s\le n$ as assumed in the proposition. 

If $n\ge 6$, we can simplify (\ref{quantunifsph0}) with 
\bea
\label{quantunifsph} 
\bP \left( \|X^t_{\tilde{I}}v\|_\infty \ge 80 \ \frac{\log(p_0)}{p_0} \right) & \le & p_0^{-n}
\eea

\subsubsection{Extracting a well conditionned submatrix of $X_{\tilde{I}}$}
The method for extracting $X_I$ from $X_{\tilde{I}}$ uses random column selection. For this 
purpose, we will need to control the coherence and the norm of $X_{\tilde{I}}$.

{\bf Step 1: The coherence of $X_{\tilde{I}}$}. Let us define the spherical cap 
\bean 
\cal C(v,h) & = & \left\{w \in \R^n \mid \la v,w\ra \ge h \right\}.
\eean 
The area of $\cal C(v,h)$ is given by 
\bean 
Area\left(\cal C(v,h)\right) & = &  Area (\cal S(0,1))  \int_0^{2h-h^2} t^{\frac{n-1}2}(1-t)^{\frac12} dt.
\eean
Thus, the probability that a random vector $w$ with Haar measure on the unit sphere $\cal S(0,1)$ falls into 
the spherical cap $\cal C(v,h)$ is given by 
\bean 
\bP \left(w\in \cal C(v,h) \right) & = & \frac{\cal C(v,h)}{\cal S(0,1)} \\
& = & \frac{\int_0^{2h-h^2} t^{\frac{n-1}2}(1-t)^{\frac12} dt}{\int_0^{1} t^{\frac{n-1}2}(1-t)^{\frac12} dt}. 
\eean 
The last term is the CDF of the Beta distribution. Using the fact that 
\bean 
\bP \left(X_j \in \cal C(X_{j^\prime},h) \right) & = & \bP \left(X_{j^\prime} \in \cal C(X_j,h) \right)
\eean 
the union bound, and the independence of the $X_j$'s, the probability that $X_j \in \cal C(X_{j^\prime},h)$ for some 
$(j,j^\prime)$ in $\left\{1,\ldots,p_0\right\}^2$ can be bounded as follows
\bean 
\bP \left(\cup_{j\neq j^\prime=1}^{p_0} \left\{ X_j \in \cal C(X_{j^\prime},h) \right\} \right) & = & \bP \left(\cup_{j<j^\prime=1}^{p_0} \left\{ X_j \in \cal C(X_{j^\prime},h) \right\} \right) \\
& \le & \sum_{j<j^\prime=1}^{p_0} \bP \left(\left\{ X_j \in \cal C(X_{j^\prime},h) \right\} \right) \\
& = & \sum_{j<j^\prime=1}^{p_0} \bE \left[ \bP\left(\left\{ X_j \in \cal C(X_{j^\prime},h) \right\} \mid X_{j^\prime}\right)\right] \\
& = & \frac{p_0(p_0-1)}2 \int_0^{2h-h^2} t^{\frac{n-1}2}(1-t)^{\frac12} dt.
\eean  
Our next task is to choose $h$ so that 
\bean 
\frac{p_0(p_0-1)}2 \int_0^{2h-h^2} t^{\frac{n-1}2}(1-t)^{\frac12} dt & \le & p_0^{-n}.
\eean  
Let us make the following crude approximation 
\bean 
\frac{p_0(p_0-1)}2 \int_0^{2h-h^2} t^{\frac{n-1}2}(1-t)^{\frac12} dt & \le & 
\frac{p_0^2}2  (2h)^{\frac{n-1}2} (2h-0).
\eean  
Thus, taking 
\bean 
h & \ge & \frac12 \ \exp\left( - 2\ \left(\log(p_0)+\frac{\log\left(p_0)-\log(2)\right)}{n+1}\right)\right) 
\eean
will work. Moreover, since $p_0\ge2$, we deduce that 
\bea
\label{mutilde} 
\mu (X_{\tilde{I}}) & \le & \frac12 \ p_0^{-2}
\eea 
with probability at least $1-p_0^{-n}$.

{\bf Step 2: The norm of $X_{\tilde{I}}$}.
The norm of any submatrix $X_{S}$ with $n$ rows and $\kappa s$ columns of $X$ has the following 
variational representation
\bean 
\|X_S\| & = & \max_{\stackrel{v \in \mathbb R^{n},\ \|v\|=1}{w \in \mathbb R^{\kappa s},\ \|w\|=1}}\ v^t X_S w.
\eean
We will use an easy $\epsilon$-net argument to control this norm. For any $v\in \mathbb R^n$, 
$v^t X_j$, $j\in S$ is a sub-Gaussian random variable satisfying
\bean 
\bP \left( |v^t X_j| \ge u \right) & \le & 2 \exp\left(-c n \ u^2 \right),
\eean
for some constant $c$. Therefore, using the fact that $\|w\|=1$, we have that 
\bean 
\bP \left( \left|\sum_{j\in S} v^t X_Sw \right| \ge u \right) & \le & 2 \exp\left(-c n \ u^2 \right).
\eean
Let us recall two useful results of Rudelson and Vershynin. The first one gives a bound on the 
covering number of spheres. 
\begin{prop} {\rm (\cite[Proposition 2.1]{RudelsonVershynin:CPAM09})}.
\label{net}
For any positive integer $d$, there exists an $\epsilon$-net of the unit sphere of $\R^d$ of cardinality 
\bean 
2d \left(1+\frac2{\epsilon}\right)^{d-1}.
\eean
\end{prop}
The second controls the approximation of the norm based on an $\epsilon$-net. 
\begin{prop} {\rm (\cite[Proposition 2.2]{RudelsonVershynin:CPAM09})}.
\label{rv2}
Let $\mathcal N$ be an $\epsilon$-net of the unit sphere of $\R^d$ and let 
$\mathcal N^\prime$ be an $\epsilon^\prime$-net of the unit sphere of $\R^{d^\prime}$. Then 
for any linear operator $A: \R^d \mapsto \R^{d^\prime}$, we have 
\bean 
\|A\| & \le & \frac{1}{(1-\epsilon)(1-\epsilon^\prime)} \sup_{\stackrel{v \in \mathcal N}{w \in \mathcal N^\prime}}
|v^t Aw|. 
\eean   
\end{prop}
Let $\mathcal N$ (resp. $\mathcal N^\prime$) be an $\epsilon$-net of the unit sphere of $\R^{\kappa s}$ 
(resp. of $\R^n$). On the other hand, we have that 
\bean 
\bP \left( \sup_{\stackrel{v \in \mathcal N}{w \in \mathcal N^\prime}}
|v^t Aw|\ge u \right) & \le & 2|\mathcal N||\mathcal N^\prime| \exp\left(-c n \ u^2 \right), \\
 & \le & 8 \ n \kappa s \left(1+\frac2{\epsilon}\right)^{n+\kappa s-2} \exp\left(-c n \ u^2 \right),
\eean 
which gives 
\bean 
\bP \left( \sup_{\stackrel{v \in \mathcal N}{w \in \mathcal N^\prime}}
|v^t Aw|\ge u \right) & \le & 8 \ \frac{n \kappa s \ \epsilon^2}{(2+\epsilon)^2} 
\exp\left(-\left(c n \ u^2 -(n+\kappa s) \log\left(1+\frac2{\epsilon}\right)\right)  \right).
\eean 
Using Proposition (\ref{rv2}), we obtain that 
\bean
\bP \left( \|X_S\| \ge u \right) & \le & \bP \left( \frac{1}{(1-\epsilon)^2} 
\sup_{\stackrel{v \in \mathcal N}{w \in \mathcal N^\prime}} |v^t Aw| \ge u  \right).
\eean
Thus, we obtain 
\bean
\bP \left( \|X_S\| \ge u \right) & \le & 8 \ \frac{n \kappa s \ \epsilon^2}{(2+\epsilon)^2} 
\exp\left(-\left(c n \ (1-\epsilon)^4 \ u^2-(n+\kappa s) \log\left(1+\frac2{\epsilon}\right)\right)  \right).
\eean
To conclude, let us note that 
\bean 
\bP\left(\|X_{\tilde{I}}\|\ge u \right) & \le & \bP \left( \max_{\stackrel{S \subset \{1,\ldots,p_0\}}{|S|=\kappa s}} \|X_S\| \ge u \right) \\
& \le & {p_0\choose {\kappa s}}\ 8 \ \frac{n \kappa s \ \epsilon^2}{(2+\epsilon)^2} 
\exp\left(-\left(c n \ (1-\epsilon)^4 \ u^2-(n+\kappa s) \log\left(1+\frac2{\epsilon}\right)\right)\right).
\eean 
and using the fact that 
\bean 
{p_0\choose {\kappa s}} & \le & \left(\frac{e \ p_0}{\kappa s}\right)^{\kappa s},
\eean
one finally obtains
\bean 
\bP\left(\|X_{\tilde{I}}\|\ge u \right) & \le & 8 \  
\exp\left(-\left(c n \ (1-\epsilon)^4 \ u^2-(n+\kappa s) \log\left(1+\frac2{\epsilon}\right)- 
\kappa s \log\left(\frac{e \ p_0}{\kappa s}\right) -\log\left(\frac{n \kappa s \ \epsilon^2}{(2+\epsilon)^2}\right)\right)\right).
\eean 
The right hand side term will be less than $8 p_0^{-n}$ when 
\bean 
n \log(p_0) & \le & c n \ (1-\epsilon)^4 \ u^2-(n+\kappa s) \log\left(1+\frac2{\epsilon}\right)- 
\kappa s \log\left(\frac{e \ p_0}{\kappa s}\right) -\log\left(\frac{n \kappa s \ \epsilon^2}{(2+\epsilon)^2}\right).
\eean
This happens if 
\bean 
u^2 & \ge & \frac1{c (1-\epsilon)^4} \left(n \frac{\log(p_0)}{n}+\left(1+\frac{\kappa s}{n}\right) 
\log\left(1+\frac2{\epsilon}\right)+ \frac{\kappa s}{n} 
\log\left(\frac{e \ p_0}{\kappa s}\right)+\frac1{n}\log\left(\frac{n \kappa s \ \epsilon^2}{(2+\epsilon)^2}\right)\right).
\eean
Notice that 
\bea
\label{sn} 
& &\left(1+\frac{\kappa s}{n}\right) 
\log\left(1+\frac2{\epsilon}\right)+\frac{\kappa s}{n} 
\log\left(\frac{e}{\kappa s}\right)
+\frac1{n}\log\left(\frac{n \kappa s \ \epsilon^2}{(2+\epsilon)^2}\right) \\
& & \hspace{3cm} \le 
\left(1+C_\kappa \right) 
\log\left(1+\frac2{\epsilon}\right)+C_\kappa+\frac1{n}\log\left(\frac{C_\kappa n^2}{4}\right), \nonumber \\
& & \hspace{3cm} \le K_{\epsilon} \ \frac{6}{\sqrt{2\pi}}, \nonumber
\eea 
since $n\ge 1$. Now, since
\bean 
\frac{6}{\sqrt{2\pi}} & \le \log(p_0) \le & \frac{n+\kappa s}{n} \log(p_0),
\eean 
we finally obtain 
\bea
\label{normbnd}
\bP\left(\|X_{\tilde{I}}\|\ge \frac{1+K_\epsilon} 
{c (1-\epsilon)^4}\frac{n+\kappa s}{n} \log(p_0) \right) & \le &  \frac8{p_0^{n}}.
\eea

{\bf Step 3}. We will use the following lemma on the distance to identity of randomly 
selected submatrices. 
\begin{lemm} 
\label{submat}
Let $r\in(0,1)$. Let $n$, $\kappa$ and $s$ satisfy conditions (\ref{n}) and (\ref{kappa}) assumed in 
Proposition \ref{gamtheo}.
Let $\Sigma\subset \left\{1,\ldots,\kappa s\right\}$ 
be a random support with uniform distribution on index sets with cardinal $s$.
Then, with probability greater than or equal to $1-9 \ p_0^{-n}$ on $X$, the following bound holds:
\bea \label{sing}
\bP \left(\|X^t_{\Sigma}X_{\Sigma}-\Id_{s} \|\ge r \mid X\right) & < & 1.
\eea
\end{lemm}
\begin{proof}
See Appendix.
\end{proof}

Taking $r=1-\rho_-$, we conclude from Lemma \ref{submat} that, for any $s$ satisfying (\ref{s}), 
there exists a subset $\tilde{\tI}$ of $\tilde{I}$ with cardinal $s$ such that 
\bean 
\sigma_{\min}\left(X_{\tilde{\tI}}\right)\ge \rho_-.
\eean

\subsubsection{The supremum over an $\epsilon$-net}
Recalling Proposition \ref{net}, there exists an $\epsilon$-net $\cal N$ covering the unit sphere in $\R^n$ with cardinal 
\bean 
|\mathcal N | & \le & 2n\left( 1+\frac2{\epsilon}\right)^{n-1}.
\eean
Combining this with (\ref{quantunifsph}), we have that 
\bea
\nonumber
& \bP \left( 
\sup_{v\in \cal N} \inf_{I\subset \cal S_{s,\rho_-}} \left\|X_I^tv \right\| \ge 
\frac{8\ \sqrt{\pi}}{e^{2\ln(2)}} \ \frac{n\ (n-2)^{1/2}}{(n-3)^{3/2}} \ \frac{\log(p_0)}{p_0} \right) \\ 
\label{netcontrol}
& \le 2n\left( 1+\frac2{\epsilon}\right)^{n-1} \ p_0^{-n}+9\ p^{-n}.
\eea

\subsubsection{From the $\epsilon$-net to the whole sphere}

For any $v^\prime$, one can find $v\in \cal N$ with $\|v^\prime-v\|_2\le \epsilon$. Thus, we have 
\bea
\nonumber
\|X_I^tv^\prime \|_{\infty} & \le & \|X_I^t v \|_{\infty}+ \|X_I^t (v^\prime-v) \|_{\infty} \\
\nonumber 
& \le & \|X_I^t v \|_{\infty}+ \max_{j\in I} |\la X_j,(v^\prime-v) \ra |\\
\nonumber 
& \le & \|X_I^t v \|_{\infty}+ \max_{j\in I} \| X_j\|_{2} \|v^\prime-v\|_2 \\
\label{sph} & \le & \|X_I^t v \|_{\infty}+ \epsilon.
\eea
Taking 
\bean 
\epsilon & = & 80\ \frac{\log(p_0)}{p_0},
\eean 
we obtain from (\ref{sph}) and (\ref{netcontrol}) that 
\bean 
& \bP \left( 
\sup_{\|v\|_2=1} \inf_{I\subset \cal S_{s,\rho_-}} \left\|X_I^tv \right\| \ge 
80\ \frac{\log(p_0)}{p_0}\right) \\ & \le 20 \ n\left( 1+\frac{p_0}{80\log(p_0)}\right)^{n-1} \ p_0^{-n}+9\ p_0^{-n}
\eean 
and thus,  
\bean 
& \bP \left(\sup_{\|v\|_2=1} \inf_{I\subset \cal S_{s,\rho_-}} \left\|X_I^tv \right\| \ge 
80\ \frac{\log(p_0)}{p_0}\right) \\ & \le 5 \ \frac{n}{p_0\ \log(p_0)^{n-1}}+9\ p_0^{-n},
\eean
for $p_0\ge \exp(6/\sqrt{2\pi})$.

\subsection{Proof of Theorem \ref{main}} 

\subsubsection{Optimality conditions}
The optimality conditions for the LASSO are given by 
\bea 
\label{opt}
-X^t(y-X\hat{\beta})+\lb g=0
\eea
for some $g\in \partial \left(\|\cdot\|_1\right)_{\hat{\beta}}$. Thus, 
we have 
\bea 
X^tX(\hat{\beta}-\beta)=X^t \epsilon-\lb g.
\eea
from which one obtains that, for any index set $\mI \subset \{1,\ldots,p\}$ with cardinal $s$,
\bea 
\label{optinf1}
\left\|X_{\mI}^tX(\beta-\hat{\beta})\right\|_{\infty} & \le & \lb+\left\|X_{\mI}^t\epsilon \right\|_{\infty},
\eea 
\subsubsection{The support of $\hat{\beta}$}
\label{supp}
As is well known, even when the solution of the LASSO optimization problem is not unique, 
there always exists a vector $\hat{\beta}$ whose support has cardinal $n$. 

\subsubsection{A bound on $\|X_\mI^t X_S(\beta_S-\hat{\beta}_S)\|_{\infty}$}

The argument is divided into three steps. 

{\em First step}. Equation (\ref{optinf1}) implies that 
\bea 
\label{optinf2}
\left\|X_{\mI}^tX_S(\beta_S-\hat{\beta}_S)\right\|_{\infty} & \le & \lb + \left\|X_{\mI}^t\epsilon \right\|_{\infty}
+\left\|X_{\mI}^tX_{\Sc}(\beta_{\Sc}-\hat{\beta}_{\Sc})\right\|_{\infty}.
\eea 

{\em Second step}. We now choose $\mI$ as a solution of the following problem
\bean 
\vartheta & = &  \min_{\stackrel{I \subset \{1,\ldots,p\}}{|I| = s}} 
\max_{j\in I} |\langle X_j,X_{\Sc} (\beta_{\Sc}-\hat{\beta}_{\Sc})\rangle|
\eean
subject to 
\bean 
\sigma_{\min}(X_{I}) & \ge & \rho_-.
\eean
By Definition \ref{defgam}, 
\bean  
\vartheta & \le & \gamma_{s,\rho_-}(X) \ \| X_{\Sc}(\beta_{\Sc}-\hat{\beta}_{\Sc})\|_2
\eean
and thus, 
\bean 
\vartheta & \le & \gamma_{s,\rho_-}(X) \ \sigma_{\max}(X_{\Sc}) \| \beta_{\Sc}-\hat{\beta}_{\Sc}\|_2 \\
& \le & \gamma_{s,\rho_-}(X) \ \sigma_{\max}(X_{\Sc}) \| \beta_{\Sc}-\hat{\beta}_{\Sc}\|_1 
\eean
which gives 
\bea
\label{varth}
\vartheta & \le & \gamma_{s,\rho_-}(X) \ \sigma_{\max}(X_{\Sc}) \left(\| \beta_{\Sc}\|_1+\|\hat{\beta}_{\Sc}\|_1\right). 
\eea 

{\em Third step}. Combining (\ref{optinf2}) and (\ref{varth}), we obtain 
\bean
\left\|X_{\mI}^tX_S(\beta_S-\hat{\beta}_S)\right\|_{\infty} & \le & \lb + \left\|X_{\mI}^t\epsilon \right\|_{\infty}
+\gamma_{s,\rho_-}(X) \ \sigma_{\max}(X_{\Sc}) \left(\| \beta_{\Sc}\|_1+\|\hat{\beta}_{\Sc}\|_1\right).
\eean 
Using the fact that 
\bea 
\left\|X_{\mI}^t\epsilon \right\|_{\infty} & \le & \left\|X^t\epsilon \right\|_{\infty}
\eea
and since 
\bean 
\bP \left( \left\|X^t\epsilon \right\|_{\infty} \ge \sg \sqrt{2\alpha \ \log(p)+\log(2p)} \right) & \le & p^{-\alpha},
\eean
we obtain that 
\bea
\nonumber
\left\|X_{\mI}^tX_S(\beta_S-\hat{\beta}_S)\right\|_{\infty} & \le & \lb + \sg \sqrt{(2\alpha+1) \ \log(p)+\log(2)} \\
& & \hspace{1cm}+\gamma_{s,\rho_-}(X) \ \sigma_{\max}(X_{\Sc}) \left(\| \beta_{\Sc}\|_1+\|\hat{\beta}_{\Sc}\|_1\right)
\label{bdlinf}
\eea 
with probability greater than $1-p^{\alpha}$. 

\subsubsection{A basic inequality} The definition of $\hat{\beta}$ gives 
\bean
\frac12 \|y-X\hat{\beta}\|_2^2+\lb \|\hat{\beta}\|_1 & \le & \frac12 \|y-X\beta\|_2^2+\lb \|\beta\|_1 
\eean
Therefore, we have that 
\bean
\frac12 \|\epsilon-X(\hat{\beta}-\beta)\|_2^2+\lb \|\hat{\beta}\|_1 & \le & \frac12 \|\epsilon\|_2^2+\lb \|\beta\|_1
\eean
which implies that 
\bean 
\frac12 \|X(\hat{\beta}-\beta)\|_2^2 & \le & \langle \epsilon, X_S(\hat{\beta}_S-\beta_S)\rangle +
\langle \epsilon, X_{\Sc}(\hat{\beta}_{\Sc}-\beta_{\Sc})\rangle \\
& & +\lb \left(\|\beta_S\|_1-\|\hat{\beta}_S\|_1\right)-\lb \|\hat{\beta}_{\Sc}\|_1+\lb \|\beta_{\Sc}\|_1.
\eean
This can be further written as 
\bea
\label{basic}
\frac12 \|X(\hat{\beta}-\beta)\|_2^2 & \le & \langle \epsilon, X_S(\hat{\beta}_S-\beta_S)\rangle +
\langle X_{\Sc}^t\epsilon, \hat{\beta}_{\Sc}-\beta_{\Sc}\rangle \\
& & +\lb \left(\|\beta_S\|_1-\|\hat{\beta}_S\|_1\right)-\lb \|\hat{\beta}_{\Sc}\|_1+\lb \|\beta_{\Sc}\|_1.
\eea

\subsubsection{Control of $\langle \epsilon,X_S(\hat{\beta}_S-\beta_S)\rangle$}

The argument is divided into two steps. 

{\em First step}. We have 
\bean 
\langle \epsilon,X_S(\hat{\beta}_S-\beta_S)\rangle 
& = & \langle X_S^t\epsilon,\hat{\beta}_S-\beta_S\rangle \\
& \le & \left\|X_S^t\epsilon \right\|_{\infty} \|\hat{\beta}_S-\beta_S\|_1 \\
& \le & \sqrt{s} \ \left\|X_S^t\epsilon \right\|_{\infty} \|\hat{\beta}_S-\beta_S\|_2
\eean
and, using the fact that $\sigma_{\min}(X_{\mI})\ge\rho_-$, 
\bean    
\langle \epsilon,X_S(\hat{\beta}_S-\beta_S)\rangle 
& \le & \frac{s}{\rho_- \ \sigma_{\min}(X_S)}\ \left\|X_S^t\epsilon \right\|_{\infty} 
\|X_{\mI}^t X_S(\hat{\beta}_S-\beta_S)\|_{\infty}.
\eean

{\em Second step}. Since the columns of $X$ have unit $\ell_2$-norm, we have 
\bean 
\bP \left( \left\|X_S^t\epsilon \right\|_{\infty} \ge \sg \sqrt{2\alpha \ \log(p)+\log(2s)} \right) & \le & p^{-\alpha},
\eean
which implies that 
\bea
\label{noise}
\langle \epsilon,X_S(\hat{\beta}_S-\beta_S)\rangle 
& \le & \frac{s \ \sg \sqrt{2\alpha \ \log(p)+\log(2s)}}{\rho_- \ \sigma_{\min}(X_S)}
\ \|X_{\mI}^t X_S\left(\hat{\beta}_S-\beta_S\right)\|_{\infty}
\eea
with probability at least $1-p^{-\alpha}$. 

\subsubsection{Control of $\langle X_{\Sc}^t\epsilon, \hat{\beta}_{\Sc}-\beta_{\Sc}\rangle$}
We have 
\bea 
\label{Sc}
\langle X_{\Sc}^t\epsilon, \hat{\beta}_{\Sc}-\beta_{\Sc}\rangle & \le & 
\left\|X_{\Sc}^t\epsilon \right\|_{\infty} \|\hat{\beta}_{\Sc}-\beta_{\Sc}\|_1. 
\eea
On the other hand, we have 
\bean 
\bP \left( \left\|X_{\Sc}^t\epsilon \right\|_{\infty} \ge \sg \sqrt{2\alpha \ \log(p)+\log(2(p-s))} \right) & \le & p^{-\alpha},
\eean
which, combined with (\ref{Sc}), implies that  
\bean 
\langle X_{\Sc}^t\epsilon, \hat{\beta}_{\Sc}-\beta_{\Sc}\rangle & \le & 
\sg \sqrt{2\alpha \ \log(p)+\log(2(p-s))} \left( \|\hat{\beta}_{\Sc}\|_1+\|\beta_{\Sc}\|_1\right)
\eean
with probability at least $1-p^{-\alpha}$.
 
\subsubsection{Control of $\|\beta_S\|_1-\|\hat{\beta}_S\|_1$}
The subgradient inequality gives 
\bean 
\|\hat{\beta}_S\|_1-\|\beta_S\|_1 & \ge & \langle {\rm sign}(\beta_S),\hat{\beta}_S-\beta_S\rangle. 
\eean
We deduce that  
\bean
\|\beta_S\|_1-\|\hat{\beta}_S\|_1 & \le & \left\|-{\rm sign}(\beta_S)\right\|_{\infty} \|\hat{\beta}_S-\beta_S\|_1 \\
& \le & \frac{\sqrt{s}}{\rho_- \ \sigma_{\min}(X_S)} \|X_{\mI}^tX_S (\hat{\beta}_S-\beta_S)\|_2
\eean
which implies 
\bea
\label{sign}
\|\beta_S\|_1-\|\hat{\beta}_S\|_1 & \le & \frac{s}{\rho_- \ \sigma_{\min}(X_S)} 
\ \|X_{\mI}^tX_S (\hat{\beta}_S-\beta_S)\|_{\infty}.
\eea

\subsubsection{Summing up}
Combining (\ref{basic}) with (\ref{noise}), (\ref{sign}) and (\ref{bdlinf}), the union bound gives that, with probability 
$1-3p^{-\alpha}$, 
\bean
\frac12 \|X(\hat{\beta}-\beta)\|_2^2 & \le & \frac{s}{\rho_- \ \sigma_{\min}(X_S)}
\left(\sg \sqrt{2\alpha \ \log(p)+\log(2s)}+\lb \right) \Bigg(\lb + \sg \sqrt{(2\alpha+1) \ \log(p)+\log(2)} \\
& & \hspace{1cm} 
+\gamma_{s,\rho_-}(X) \ \sigma_{\max}(X_{\Sc}) \left(\| \beta_{\Sc}\|_1+\|\hat{\beta}_{\Sc}\|_1\right)\Bigg) \\
& & +\sg \sqrt{2\alpha \ \log(p)+\log(2(p-s))} \left(\| \beta_{\Sc}\|_1+\|\hat{\beta}_{\Sc}\|_1\right) \\
& & +\lb\left(\| \beta_{\Sc}\|_1-\|\hat{\beta}_{\Sc}\|_1\right) 
\eean
which gives,  
\bean
\frac12 \|X(\hat{\beta}-\beta)\|_2^2 & \le & s \ \frac{ \lb + \sg \sqrt{(2\alpha+1) \ \log(p)+\log(2)}}{\rho_- \ \sigma_{\min}(X_S)}
\left(\sg \sqrt{2\alpha \ \log(p)+\log(2s)}+\lb \right)  \\
& & +\Bigg(\frac{s }{\rho_- \ \sigma_{\min}(X_S)}
\left(\sg \sqrt{2\alpha \ \log(p)+\log(2s)}+\lb \right)\gamma_{s,\rho_-}(X) \ \sigma_{\max}(X_{\Sc}) \\
& & \hspace{1cm}  +\sg \sqrt{2\alpha \ \log(p)+\log(2(p-s))}-\lambda\Bigg) \|\hat{\beta}_{\Sc}\|_1 \\
& & +\Bigg(\frac{s }{\rho_- \ \sigma_{\min}(X_S)}
\left(\sg \sqrt{2\alpha \ \log(p)+\log(2s)}+\lb \right)\gamma_{s,\rho_-}(X) \ \sigma_{\max}(X_{\Sc}) \\
& & \hspace{1cm}  +\sg \sqrt{2\alpha \ \log(p)+\log(2(p-s))}+\lambda\Bigg) \|\beta_{\Sc}\|_1. 
\eean
Using the assumption that $s\le \nu n$, we obtain 
\bean
\frac12 \|X(\hat{\beta}-\beta)\|_2^2 & \le & s \ \frac{ \lb + \sg \sqrt{(2\alpha+1) \ \log(p)+\log(2)}}{
\rho_- \ \sigma_{\min}(X_S)}
\left(\sg \sqrt{2\alpha \ \log(p)+\log(2\nu n)}+\lb \right)  \\
& & +\Bigg(\frac{\nu n }{\rho_- \ \sigma_{\min}(X_S)}
\left(\sg \sqrt{2\alpha \ \log(p)+\log(2\nu n)}+\lb \right)\gamma_{s,\rho_-}(X) \ \sigma_{\max}(X_{\Sc}) \\
& & \hspace{1cm}  +\sg \sqrt{(2\alpha+1) \ \log(p)+\log(2)}-\lambda\Bigg) \|\hat{\beta}_{\Sc}\|_1 \\
& & +\Bigg(\frac{\nu n}{\rho_- \ \sigma_{\min}(X_S)}
\left(\sg \sqrt{2\alpha \ \log(p)+\log(2\nu n)}+\lb \right)\gamma_{s,\rho_-}(X) \ \sigma_{\max}(X_{\Sc}) \\
& & \hspace{1cm}  +\sg \sqrt{2\alpha \ \log(p)+\log(2(p))}+\lambda\Bigg) \|\beta_{\Sc}\|_1. 
\eean
Since, as recalled in Section \ref{supp}, the support of $\hat{\beta}$ has cardinal less than or equal 
to $n$, we have  
\bean 
\sigma_{\max}(X_{\Sc}) & \le & \max_{\stackrel{T\subset \{1,\ldots,p\}}{|T|\le n}}\sigma_{\max}(X_T),
\eean 
and the proof is completed. 

\section{A simple trick when $\gamma_{s,\rho_-}$ is unknown: appending a random matrix}
\label{GaussAppend}
We have computed the index $\gamma_{s,\rho_-}$ for the random matrix with independent 
columns uniformly distributed on the unit sphere of $\mathbb R^n$ in Theorem \ref{gamtheo}. 
The goal of this section is to show that this result can be used in a simple trick 
in order to obtain prediction bounds similar to \cite[Theorem 2.1]{CandesPlan:AnnStat09}
without conditions on the design matrix $X$. 

This idea is of course to use Theorem \ref{main} above. However, the values of $\sigma_{\min}(X_S)$ and 
$\sigma_{\max}(X_{\Sc})$ are of course usually not known ahead of time and we have to provide easy to 
compute bounds for these quantities. The coherence $\mu(X)$ can be used for this purpose. Indeed, 
for any positive integer $t\le p$ and any $T\subset \{1,\ldots,p\}$ with $|T|=t$, we have 
\bean 
\mu(X) & = & \|X^tX-I\|_{1,1} \\
& = & \max_{\|w\|_\infty =1} \max_{\|w^\prime\|_1=1} w^t(X^tX-I)w^\prime \\
& \ge & \frac1{\sqrt{t}} 
\max_{\stackrel{\|w\|_2=1}{\|w\|_0=t}} \max_{\stackrel{\|w^\prime\|_2=1}{\|w^\prime\|_0=t}} 
w^t(X^tX-I)w^\prime. 
\eean 
Thus, we obtain that 
\bean 
1-\mu(X)\sqrt{t} & \le \sigma_{\min}(X_T) \le \sigma_{\max}(X_T) \le & 1+\mu(X)\sqrt{t}. 
\eean 
However, the lower bound on $\sigma_{\min}(X_S)$ obtained in this manner may not 
be accurate enough. More precise, polynomial time computable, bounds have been 
devised in the litterature. The interested reader can find a very useful Semidefinite relaxation 
of the problem of finding the worst possible value of $\sigma_{\min}(X_T)$ over all 
subsets $T$ of $\{1,\ldots,p\}$ with a given cardinal (related to the Restricted Isometry 
Constant) in \cite{DAspremontBachElGhaoui:JMLR08}. 

Assuming we have a polynomial time computable a priori bound $\sigma_{\min}^*$ on
$\sigma_{\min}(X_T)$ (resp. $\sigma_{\max}^*$ on 
$\max_{\stackrel{T\subset \{1,\ldots,p\}}{|T|\le n}}\sigma_{\max}(X_T)$), 
our main result for the case of general design matrices is the 
following theorem. 
\begin{theo}
Let $X$ be an matrix in $\R^{n\times p}$ with $\ell_2$-normalized columns and 
let $X_0$ be a random matrix with independent columns uniformly distributed on the unit 
sphere of $\R^n$. Let $X_\sharp$ denote the matrix corresponding to the concatenation of $X$ and $X_0$, 
i.e. $X_\sharp=[X,X_0]$. Let $\hat{\beta}_\sharp$ denote the LASSO estimator 
with $X$ replaced with $X_\sharp$ in (\ref{lasso}).  
Let $\rho_-\in (0,1)$. Let $\nu$ be a positive real. Assume that 
$p_0$ is such that 
\bea
\label{indconcat}
80\ \frac{\log(p_0)}{p_0} & < & L \ \rho_- \frac{\sigma_{\min}^*}{\nu n \ \sigma_{\max}^*}
\eea 
for some $L\in (0,1)$. Assume moreover that $p_0$ is sufficiently large 
so that the second inequality in (\ref{n}) is satisfied. 
Assume that $\beta$ has support  $S$ with cardinal $s$ and that  
\bean
\lb & \ge &  \sg  \left(B^\prime_{X,\nu,\rho_-} \ \sigma_{\max}^*
\sqrt{2\alpha \ \log(p+p_0)+\log(2\nu n)}+ \sqrt{(2\alpha+1) \ \log(p+p_0)+\log(2)}
\right)  
\eean 
with 
\bea 
B^\prime_{X,\nu,\rho_-} & = & 
\frac{\nu n \ \gamma_{\nu n,\rho_-}(X)}{\rho_- \ \sigma_{\min}^*- \nu n \ \gamma_{\nu n,\rho_-}(X)\ 
\sigma_{\max}^*}.
\eea
Assume that $s$ satisfies the first inequality in (\ref{n}) and that $s\le \nu n$.
Then, with probability greater than $1-p^{-\alpha}-9 p_0^{-n}-20 \ \frac{n}{\log(p_0)^{n-1}} p_0^{-1}$, we have 
\bea
\frac12 \|X(\hat{\beta}_{\#}-\beta)\|_2^2 & \le & s \ C^\prime_{n,p,\rho_-,\alpha,\nu,\lb} 
\eea
with 
\bean
C^\prime_{n,p,\rho_-,\alpha,\nu,\lb} & = & \frac{\lb + \sg \sqrt{(2\alpha+1) \ \log(p+p_0)+\log(2)}}{
\rho_- \ \sigma_{\min}^*}
\left(\sg \sqrt{2\alpha \ \log(p+p_0)+\log(2\nu n)}+\lb \right)
\eean
\end{theo}

\begin{proof}
Since the index $\gamma_{s,\rho-}$ does not increase after appending a matrix with $\ell_2$-
normalized columns, the matrix $X_{\#}$ has at most the same index as that of $X_0$. Then 
(\ref{indconcat}) ensures that the index $\gamma_{s,\rho-}(X_{\#})$ is sufficiently 
small. The rest of the proof is identical to the proof of Theorem \ref{main}. 
\end{proof}

\appendix 

\section{Proof of Lemma \ref{submat}}
For any index set $S\subset \{1,\ldots,\kappa s\}$ with cardinal $s$, 
define $R_S$ as the diagonal matrix with 
\bean 
(R_{S})_{i,i} & = & 
\begin{cases}
1 \textrm{ if } i\in S, \\
0 \textrm{ otherwise.}
\end{cases}
\eean
Notice that we have 
\bean 
\left\|X_S^tX_S-I\right\| & = & \left\|R_SHR_S\right\| 
\eean
with $H=X^tX-I$. In what follows, $R_\delta$ simply denotes a diagonal 
matrix with i.i.d. diagonal components $\delta_j$, $j=1,\ldots,\kappa s$ 
with Bernoulli $B(1,1/\kappa)$ distribution. Let $R^{\prime}$ be an independent copy of $R$.
Assume that $S$ is drawn uniformly at random among index sets of $\{1,\ldots,\kappa s\}$
with cardinal $s$. By an easy Poissonization argument, similar to \cite[Claim $(3.29)$ p.2173]{CandesPlan:AnnStat09},
we have that 
\beq \label{poisse}
\bP \left(\|R_sHR_s\|\ge r \right) \ \le \ 2\ \bP \left(\|RHR\|\ge r\right),
\eeq
and by Proposition 4.1 in \cite{ChretienDarses:SPL12}, we have that  
\bea \label{dec}
\bP\left(\|RHR\|\ge r\right) & \le & 36\ \bP\left(\|RHR^\prime\|\ge r/2 \right).
\eea
In order to bound the right hand side term, we will use \cite[Proposition 4.2]{ChretienDarses:SPL12}. 
Set $r^\prime=r/2$. Assuming that $\kappa \frac{{r^\prime}^2}{e} \ge u^2 \ge \frac{1}{\kappa}\|X\|^4$ and 
$v^2\ge \frac{1}{\kappa}\|X\|^2$, the right hand side term can be bounded from above as follows: 
\bea\label{inv_bound}
\bP \left(\|RHR'\|\ge r^\prime \right) & \le & 3 \ \kappa s\ \mathcal  V(s,[r^\prime,u,v]),
\eea
with
\bean
\mathcal  V(s,[r^\prime ,u,v]) & = & \left(e\frac{1}{\kappa} \frac{u^2}{{r^\prime}^2} \right)^{\frac{{r^\prime}^2}{v^2}} 
+\left(e \frac{1}{\kappa}\frac{\|M\|^4}{u^2} \right)^{u^2/\|M\|^2} 
+\left(e \frac{1}{\kappa}\frac{\|M\|^2}{v^2} \right)^{v^2/\mu(M)^2}.
\eean
Using (\ref{mutilde}) and (\ref{normbnd}), we deduce that with probability at least 
$1-8 p_0^{-n}-p_0^{-n}$, we have 
\bean
\mathcal  V(s,[r^\prime,u,v]) & = & \left(e\frac{1}{\kappa} \frac{u^2}{{r^\prime}^2} \right)^{\frac{{r^\prime}^2}{v^2}} 
+\left(e \frac{1}{\kappa}\frac{\left(\frac{1+K_\epsilon}{c(1-\epsilon)^4}\frac{n+\kappa s}{n} \log(p_0)\right)^4}{u^2} \right)
^{\frac{u^2}{\left(\frac{1+K_\epsilon}{c(1-\epsilon)^4}\frac{n+\kappa s}{n} \log(p_0)\right)^2}}  \\
& & \hspace{.5cm} +\left(e \frac{1}{\kappa}
\frac{\left(\frac{1+K_\epsilon}{c(1-\epsilon)^4}\frac{n+\kappa s}{n} \log(p_0)\right)^2}{v^2} \right)^{\frac{v^2}{\frac12 \ p_0^{-2}}}.
\eean
Take $\kappa$, $u$ and $v$ such that 
\bean 
v^2 & = &  {r^\prime}^2 \ \frac1{\log(C_\kappa\ n)}  \\
\\
u^2 & = &  C_{\cal V}\ \left( \frac{1+K_\epsilon}{c(1-\epsilon)^4}\frac{n+\kappa s}{n} \log(p_0)\right)^2,\\
\\
\kappa & \ge & e^3 \ \frac{C_{\cal V}}{{r^\prime}^2} \ \left( \frac{1+K_\epsilon}{c(1-\epsilon)^4}\frac{n+\kappa s}{n} \log(p_0)\right)^2 \\
\eean
for some $C_{\cal V}$ possibly depending on $s$. Since $\kappa s\le C_{\kappa} n$,
this implies in particular that 
\bea 
\label{kap}
\kappa & \ge & e^3 \ \frac{C_{\cal V}}{{r^\prime}^2} \ \left( \frac{(1+K_\epsilon)(1+C_{\kappa})}{c(1-\epsilon)^4} \log(p_0)\right)^2.
\eea
Thus, we obtain that 
\bean
\mathcal  V(s,[r^\prime,u,v]) & = & \left(\frac{1}{e^2} \right)^{\log(C_\kappa n)} 
+\left(\frac{{r^\prime}^2}{e^2 \ C_{\cal V}^2} \right)^{C_{\cal V}}
+\left(\frac{\log(C_\kappa n)}{e^2\ C_{\cal V}} \right)^{\frac{2 {r^\prime}^2\ p_0^2}{\log(C_\kappa n)}}.
\eean
Using (\ref{poisse}), (\ref{dec}) and (\ref{inv_bound}), we obtain that 
\bean 
\bP \left(\|R_sHR_s\|\ge r^\prime \right) & \le & 2\times 36 \times 3 \times \kappa s 
\left(\left(\frac{1}{e^2} \right)^{\log(C_\kappa n)} 
+\left(\frac{{r^\prime}^2}{e^2\ C_{\cal V}^2} \right)^{C_{\cal V}}
+\left(\frac{\log(C_\kappa n)}{e^2\ C_{\cal V}} \right)^{\frac{2 {r^\prime}^2\ p_0^2}{\log(C_\kappa n)}}\right).
\eean 
Take 
\bea
\label{CV}
C_{\cal V}  & = & \log(C_\kappa n)
\eea
and, since $p_0>1$ and $r \in (0,1)$, we obtain 
\bea
& & \bP \left(\|R_sHR_s\|\ge r^\prime \right) \nonumber \\
\label{probe}
& & \hspace{2cm} \le 2\times 36 \times 3 \times \kappa s 
\left(\left(\frac{1}{e^2} \right)^{\log(C_\kappa n)} 
+\left(\frac{{r^\prime}^2}{e^2 \ \log^2(C_\kappa n)} \right)^{\log(C_\kappa n)}+\left(\frac{1}{e^2} \right)^{\frac{2 {r^\prime}^2\ p_0^2}{\log(C_\kappa n)}}\right).
\eea 
Replace $r^\prime$ by $r/2$. Since it is assumed that $n \ge \exp(r/2)/C_{\kappa}$ and $p_0\ge \sqrt{2} \log(C_{\kappa}n)/r$, 
it is sufficient to impose that 
\bean 
C_\kappa^2 n^2  & \ge & \left(2\times 36 \times 3 \times \kappa s\times 3\right)^{\frac1{\log(e^2)}},
\eean
in order for the right hand side of (\ref{probe}) to be less than one. 
Since $\kappa s\le C_{\kappa}n$, it is sufficient to impose that 
\bean 
C_\kappa^2 n^2  & \ge & 2\times 36 \times 3 \times \ C_\kappa n\times 3,
\eean 
or equivalently, 
\bean 
C_\kappa n  & \ge & 2\times 36 \times 3 \times 3.
\eean 
This is implied by (\ref{n}) in the assumptions. 
On the other hand, combining (\ref{kap}) and (\ref{CV}) implies that one can take 
\bean 
\kappa & = & \frac{4e^3}{r^2} \ \left( \frac{(1+K_\epsilon)(1+C_{\kappa})}{c(1-\epsilon)^4}\right)^2 \log^2(p_0) \log(C_\kappa n),
\eean 
which is nothing but (\ref{kappa}) in the assumptions.


\begin{thebibliography}{1}

\bibitem{Akaike:IEEEAC74} Akaike, H., 
A new look at the statistical model identification, IEEE Trans. Automat. 
Control, (1974) 19, 716--723. 

\bibitem{BickelRitovTsybakov:AnnStat09}
Bickel, P. J., Ritov, Y., Tsybakov, A. B., Simultaneous analysis of lasso and Dantzig selector. Ann. Statist. 37 (2009), no. 4, 1705--1732.

\bibitem{BuneaTsybakovWegkamp:EJS07} Bunea, F., Tsybakov, A. and Wegkamp, M., 
Sparsity oracle inequalities for the LASSO, 2007, The Electronic Journal of Statistics, 169 -- 194. 

\bibitem{BuhlmannVanDeGeer:Springer11} B\"uhlmann, P., van de Geer, S., Statistics for High-Dimensional Data, Methods, Theory and Applications, Series: Springer Series in Statistics (2011).

\bibitem{BuhlmannVandeGeer:EJS09} van de Geer, S. and B\"uhlmann, P., On the conditions used to prove oracle results for the Lasso. Electronic Journal of Statistics (2009) 3, 1360--1392.




\bibitem{CandesPlan:AnnStat09}
Cand\`es, E. J. and Plan, Y. Near-ideal model selection by $\ell_1$ minimization.
 Ann. Statist. 37 (2009),  no. 5A, 2145--2177.



\bibitem{ChretienDarses:SPL12} Chr\'etien, S. and Darses, S., Invertibility of random submatrices via tail decoupling and a Matrix Chernoff Inequality, Stat. and Prob. Lett., (2012), 82, no. 7, 1479--1487.

\bibitem{ChretienDarses:ArXiv2012ieeeit} Chr\'etien, S. and Darses, S., Sparse recovery with unknown variance: a LASSO-type approach, IEEE Trans. Info. Th., to appear. 






\bibitem{DalalyanTsybakov:JCSS12}
Dalalyan, A. and Tsybakov, A., 
Sparse Regression Learning by Aggregation and Langevin Monte-Carlo, J. Comput. System Sci. 78 (2012), pp. 1423-1443.

\bibitem{DAspremontBachElGhaoui:JMLR08} 
A. d'Aspremont, F. Bach and L. El Ghaoui, Optimal Solutions for Sparse Principal Component Analysis.
Journal of Machine Learning Research, 9 (2008), pp. 1269-1294.

\bibitem{Dubhashi:CUP09} Dubhashi, D. P. and Panconesi, A., Concentration of measure for the analysis of randomized algorithms. Cambridge University Press, Cambridge, 2009.

\bibitem{FosterGeorge:AnnStat94} The Risk Inﬂation Criterion for Multiple Regression. Dean P. Foster; Edward I. George. Annals of Statistics, Volume 22, Issue 4 (1994), 1947-1975.


\bibitem{Muirhead:AspectsMultAnal05} Muirhead, R., Aspects of multivariate statistical theory.
Wiley Series in Probability and Mathematical Statistics. John Wiley \& Sons, Inc., New York, 1982. xix+673 pp.

\bibitem{Qi:JIA10} Qi, F, Bounds for the ratio of two gamma functions. J. Inequal. Appl. 2010.

\bibitem{RigolletTsybakov:AnnStat11} Rigollet, P. and Tsybakov, A., 
Exponential Screening and optimal rates of sparse estimation, Ann. Statist., 39(2), 731-771. 

\bibitem{RudelsonVershynin:CPAM09} Rudelson, M. and Vershynin, R., Smallest singular value of a random rectangular matrix. 
Comm. Pure Appl. Math. 62 (2009), no. 12, 1707--1739. 

\bibitem{Schwarz:AnnStat78} Estimating the dimension of a model, The Ann. of Stat., (1978) 6, 
461--464.

\bibitem{Tibshirani:JRSSB96} Tibshirani, R. Regression shrinkage and selection via the LASSO, J.R.S.S. 
Ser. B, 58, no. 1 (1996), 267--288. 

\bibitem{Tropp:FOCM12} Tropp, J., User friendly tail bounds for sums of random matrices, 
Foundations of Computational Mathematics, 
(2012), 12, no.4, 389--434.

\bibitem{vdGeer:EJS09} van de Geer, S. and B\"uhlmann, P., On the conditions used to prove oracle results for the Lasso. Electron.
J. Stat. 3 (2009) 1360--1392.



\bibitem{Wainwright:IEEEIT09} Wainwright, Martin J., Sharp thresholds for high-dimensional and noisy sparsity recovery
 using $\ell_1$-constrained quadratic programming (Lasso). IEEE Trans. Inform. Theory  55  (2009), no. 5, 2183--2202.


\bibitem{ZhaoYu:JMLR06} Zhao, P. and Yu, B., On model selection consistency of Lasso.
J. Mach. Learn. Res. 7 (2006), 2541--2563.

\end{thebibliography}
\end{document}